\documentclass[10pt]{article}
\usepackage{amsmath,amssymb,amsfonts}
\usepackage{epsf,color}
\usepackage{fancyhdr}
\usepackage{blindtext}
\usepackage{url}
\newtheorem{theorem}{Theorem}

\date{}

\title{Large Components and Trees of Random Mappings}
\bigskip

\author{  {\sc Ljuben Mutafchiev}\thanks{Corresponding author; E-mail: Ljuben@aubg.edu}\,\,\thanks{Also at: Institute of Mathematics
and Informatics, Bulgarian Academy of Sciences,
  Sofia 1113, Bulgaria} \\ \small American University in Bulgaria,
  2700 Blagoevgrad, Bulgaria
  \and
 {\sc Steven Finch}
 \\ \small MIT Sloan School of Management, Cambridge, MA 02142, USA}
\date{}
\begin{document}
\maketitle

\begin{abstract}
Let $\mathcal{T}_n$ be the set of all mappings $T:[n]\to[n]$, where
$[n]=\{1,2,\ldots,n\}$. The corresponding graph $G_T$ of $T$, called
a functional digraph, is a union of disjoint connected components.
Each component is a directed cycle of rooted labeled trees. We
assume that each $T\in\mathcal{T}_n$ is chosen uniformly at random
from the set $\mathcal{T}_n$. The components and trees of $G_T$ are
distinguished by their size. In this paper, we compute the limiting
conditional probability ($n\to\infty$) that a vertex from the
largest component of the random graph $G_T$, chosen uniformly at
random from $[n]$, belongs to its $s$-th largest tree, where $s\ge
1$ is a fixed integer. This limit can be also viewed as an
approximation of the probability that the $s$-th largest tree of
$G_T$ is a subgraph of its largest component, which is a solution of
a problem suggested by Mutafchiev and Finch (2024).

\end{abstract}

\vspace{.5cm}



\vspace{.2cm}

\section{Notation, Statement and Proof of the Main Result}

For a positive integer $n$, let $\mathcal{T}_n$ denote the set of
all mappings $T:[n]\to [n]$, where $[n]:=\{1,2,\ldots,n\}$. It is
clear that the cardinality $|\mathcal{T}_n|$ of $\mathcal{T}_n$ is
$n^n$. A mapping $T\in\mathcal{T}_n$ corresponds to a directed graph
$G_T$, called a functional digraph, with edges $(v,T(v)), v\in [n]$,
where every vertex $v\in [n]$ has out-degree $1$. $G_T$ is a union
of disjoint connected components. Each component of $G_T$ is a
directed cycle of rooted labeled trees. We introduce the uniform
probability measure $\mathbb{P}$ on the set $\mathcal{T}_n$. That
is, we assign the probability $n^{-n}$ to each $T\in\mathcal{T}_n$.
In this way, the numerical characteristics of $G_T$ become random
variables (or, statistics in the sense of random generation of
mappings from $\mathcal{T}_n$). The notation $\mathbb{E}$ stands for
the expectation with respect to $\mathbb{P}$.

Further on, we consider the largest component $m_n=m_n(T)$ and the
$s$-th largest tree $t_{n,s}=t_{n,s}(T)$ as subgraphs of the random
graph $G_T$. We denote their vertex-sets by $V(m_n)=V(m_n(T))$ and
$V(t_{n,s})=V(t_{n,s}(T))$, respectively. The cardinalities of these
vertex-sets, called sizes of $m_n(T)$ and $t_{n,s}(T)$, will be
denoted by $\mu_n=\mu_n(T):=|V(m_n(T))|$ and
$\tau_{n,s}=\tau_{n,s}(T):=|V(t_{n,s}(T))|$.

There is a substantial probabilistic literature on random mappings.
For the variety of methods used to obtain quantitative asymptotic
estimates on random mapping statistics, we direct the reader, e.g.,
to the books \cite{K86,ABT03} and the paper \cite{FO90}. For large
$n$, some properties of the functional digraphs $G_T,
T\in\mathcal{T}_n$, are closely related to random number generators
and algorithms for integer factorization. More details on these
applications may be found in \cite[Section 4.5.4]{K98}.

Next, we need to define a two-step sampling procedure that combines
the outcomes of two random experiments, in which: (i) we select a
mapping $T\in\mathcal{T}_n$ uniformly at random (step 1) and (ii) we
select a vertex $v\in [n]$ of the graph $G_T$ uniformly at random
(step 2). We shall use the concepts of a product sample space and
product probability measure for a pair of random experiments. For
more details on this subject, we refer the reader, e.g., to
\cite[Section 1.6]{GS01}. Hence the sample space of the pair of
experiments defined by (i) and (ii) is the set product
$\mathcal{T}_n\times [n]$. Let $\mathbf{P}(.)$ denote the
probability measure on this product space. Then, for a pair of
events $A_1\subset\mathcal{T}_n$ and $A_2\subset [n]$, we set
\begin{equation}\label{def}
\mathbf{P}(A_1\times A_2):=\mathbb{P}(T\in A_1)\frac{|A_2|}{n}.
\end{equation}
Returning to the random mapping statistics $\mu_n$ and $\tau_{n,s}$,
we set in (\ref{def}) $A_1=\{T\}$, $A_2=\{v\in V(m_n(T))\}$. Thus,
as a first step of our sampling, we choose a mapping $T$ with
probability $\mathbb{P}(\{T\})=n^{-n}$ and, as a second step, we
select a vertex $v\in V(m_n(T))$ with probability
$\frac{\mu_n(T)}{n}$. Then from (\ref{def}) it follows that
$$
\mathbf{P}(\{T\}\times\{v\in V(m_n)\})
=\mathbb{P}(\{T\})\frac{\mu_n(T)}{n}.
$$
Summation over all $T$ yields
\begin{equation}\label{mu}
\mathbf{P}(v\in V(m_n)) =\frac{1}{n}\mathbb{E}(\mu_n).
\end{equation}
In the same way, for $\{v\in V(t_{n,s})\}$, we obtain
\begin{equation}\label{tau}
\mathbf{P}(v\in V(t_{n,s})) =\frac{1}{n}\mathbb{E}(\tau_{n,s}).
\end{equation}

We are now ready to state our main result in terms of the limiting
conditional probability that the randomly chosen vertex $v\in [n]$
belongs to the $s$-th largest tree, given that it lies in the
largest component. Let us just note that our study is motivated by a
problem, which we formulated in \cite[p. 13]{MF24} as follows:
"...what can be said about the probability that the largest tree is
a subgraph of the largest component of a random mapping?"

\begin{theorem}.
(i) Let
\begin{equation}\label{ps}
p_s:=lim_{n\to\infty}\mathbf{P}(\{v\in V(t_{n,s})\}|\{v\in
V(m_n)\}).
\end{equation}
Then we have
\begin{equation}\label{e}
p_s=\lim_{n\to\infty}\frac{\mathbb{E}(\tau_{n,s})}{\mathbb{E}(\mu_n)}.
\end{equation}

(ii) For any fixed integer $s$, the limit $p_s$ exists, $p_s\le 1$,
and the first four values of $p_s$ are given by
\begin{eqnarray}\label{fourp}
& & p_1= 0.6380095879\ldots, \nonumber \\
& & p_2= 0.2111140604\ldots,  \\
& & p_3= 0.1083393241\ldots, \nonumber \\
& & p_4= 0.0667422345\ldots. \nonumber
\end{eqnarray}
\end{theorem}

\begin{proof}
(i) First, we note that, for a vertex $v\in [n]$ of the random graph
$G_T$, we have $\{v\in V(t_{n,s})\}\cap\{v\in V(m_n)\}=\{v\in
V(t_{n,s})\cap V(m_n)\}$. Suppose now that a randomly chosen vertex
$v\in V(t_{n,s})\cap V(m_n)$. Hence $V(t_{n,s})\cap
V(m_n)\neq\emptyset$ and therefore, since the components of $G_T$
are disjoint, we conclude that $V(t_{n,s})\subset V(m_n)$ (that is,
$t_{n,s}$ is a subgraph of $m_n$). Thus $V(t_{n,s})\cap
V(m_n)=V(t_{n,s})$ and the numerator of the conditional probability
in (\ref{ps}) is $\mathbf{P}(\{v\in V(t_{n,s})\})$. Hence
$$
\mathbf{P}(\{v\in V(t_{n,s})\}|\{v\in V(m_n)\})
=\frac{\mathbf{P}(v\in V(t_{n,s}))}{\mathbf{P}(v\in V(m_n))}.
$$
Applying (\ref{mu}) and (\ref{tau}), we get (\ref{e}). \hfill
$\Box$

(ii) In the computation of $p_s, s=1,2,3,4$, we shall use the main
terms in the asymptotic expansions of $\mathbb{E}(\mu_n)$ and
$\mathbb{E}(\tau_{n,s})$. It is known that
\begin{equation}\label{asmu}
\mathbb{E}(\mu_n)\sim cn,\quad \text{where}\quad
c=0.7578230112\ldots,\quad n\to\infty.
\end{equation}
The asymptotic equivalence (\ref{asmu}) and the value of the
constant $c$ were first derived by Flajolet and Odlyzko
\cite[Theorem 8, p. 347]{FO90}. The constant $c$ is given in A143297
of \cite{S} and named Flajolet-Odlyzko constant in \cite[Section
5.4.2]{F03}. Gourdon obtained (\ref{asmu}) in his thesis \cite[p.
152]{G96}. He then published it in a subsequent paper
\cite[Corollary 2, p. 200]{G98}. In a more general context, the
constant $c$ appears also in \cite[Table 5.1]{ABT03}.

Similarly, for $\mathbb{E}(\tau_{n,s})$, we have
\begin{equation}\label{astau}
\mathbb{E}(\tau_{n,s})\sim c_s n,\quad n\to\infty,
\end{equation}
where the constants $c_s$ are explicitly given. This result and
numerical evaluations were obtained in Gourdon's thesis \cite[p.
188]{G96}. The case $s=1$ (the largest tree in a random mapping) was
also published in \cite[p. 202]{G98} and the value of the constant
$c_1$ is given in A271871 of \cite{S} and \cite[Section 5.4.2]{F03}.
Since $\tau_{n,s}$ does not exceed the size of the component to
which the largest tree belongs and $\mu_n$ is the maximum component
size of the random graph $G_T$, for all $n,s\ge 1$, we have
$\tau_{n,s}\le\mu_n$. Hence
\begin{equation}\label{taumu}
\mathbb{E}(\tau_{n,s})\le\mathbb{E}(\mu_n).
\end{equation}
Combining (\ref{asmu}) - (\ref{taumu}), we conclude that the limit
(\ref{e}) exists and
\begin{equation}\label{csc}
p_s=\frac{c_s}{c}\le 1.
\end{equation}
Below we list the first four values of $c_s$ from \cite[p.
188]{G96}:
\begin{eqnarray}\label{fourc}
& & c_1=0.4834983471\ldots, \nonumber \\
& & c_2=0.1599870930\ldots, \\
& & c_3=0.0821020328\ldots, \nonumber \\
& & c_4=0.0505788011\ldots. \nonumber
\end{eqnarray}
The probabilities from (\ref{fourp}) are then obtained using
(\ref{csc}), (\ref{fourc}) and the value of $c$ from
(\ref{asmu}).\hfill $\Box$
\end{proof}

\section{Concluding Remarks}

As mentioned in the previous section, our main result sheds light on
problems concerning the relationship between large components and
large trees of a random mapping graph with large number of vertices.
The case $s=1$ solves a problem suggested in \cite[p. 13]{MF24}. We
obtain the limiting probability that the largest component of a
random graph $G_T, T\in\mathcal{T}_n$, contains its $s$-th largest
tree $t_{n,s}$ as $n\to\infty$. We conclude this section with two
open problems.

{\it 1. The probability that the largest component contains two
large trees.} Consider the particular case, where the largest
component $m_n$ of a random graph $G_T$ contains the first two
largest trees $t_{n,1}$ and $t_{n,2}$. What can be said about the
probability of this event? To answer this question in a way similar
to that from Section 1, we modify the second step of the random
experiment defined there by selecting two vertices
$\{v_1,v_2\}\subset [n]$ of the graph $G_T$ with probability
${n\choose 2}^{-1}$. This will obviously change the product
probability space into $\mathcal{T}_n\times\{\{u,v\}: u,v\in [n],
u\neq v\}$. Using the new product probability measure on this
product space, we establish that the formula for the conditional
probability $\mathbb{P}(v_1\in V(t_{n,1}), v_2\in
V(t_{n,2})|\{v_1,v_2\}\subset V(m_n))$ involves the second moment of
the size $\mu_n$ of the largest component of $G_T$. We note that the
leading term in the asymptotic expansion of the variance of $\mu_n$,
as $n\to\infty$, is given in \cite[p. 152]{G96}. This result, in
combination with (\ref{asmu}), could be used to determine the
asymptotics of the second moment of $\mu_n$.

{\it 2. The probability that the $r$-th largest component contains
the $s$-th largest tree: the case $r\ge 2$}. Let $m_{n,r}$ be the
$r$-th largest component of $G_T$ and let $\mu_{n,r}$ be the
cardinality of its vertex-set $V(m_{n,r})$. In this paper, we
examine only the case $r=1$. In the proof, we essentially used the
inequality (\ref{taumu}), which is true since, for all integers
$n,s\ge 1$, $\tau_{n,s}\le\mu_{n,1}=\mu_n$ with probability $1$. It
seems that a solution of a similar problem in the case $r\ge 2$ will
encounter more technical difficulties. In fact, in this case the
probability that $\tau_{n,s}\le\mu_{n,r}$ is not $1$ and one has to
estimate it. This requires some extra computations using the
distribution functions of $\mu_{n,r}, r\ge 2$, and $\tau_{n,s}$. We
recall that, for fixed integers $r,s\ge 1$, the limiting
distribution functions of $\mu_{n,r}/n$ and $\tau_{n,s}/n$ as
$n\to\infty$ were already determined (see \cite[Section 1.13]{K86}
and \cite[p.188]{G96}, respectively). These results would be helpful
to obtain certain asymptotic estimates in the case $r\ge 2$.

\end{document}